\documentclass[letterpaper,12pt,reqno]{amsart}
\RequirePackage[utf8]{inputenc}
\usepackage[portrait,margin=3cm]{geometry}
\usepackage{amsmath,amsthm,amsfonts,amssymb,amstext}
\usepackage{graphicx}
\usepackage{color}
\usepackage{natbib}

\usepackage{pdfsync}
\usepackage{fourier}

\numberwithin{equation}{section}
\theoremstyle{plain}

\newtheorem{theorem}{Theorem}
\newtheorem*{theorem*}{Theorem}

\newtheorem{corollary}[theorem]{Corollary}

\newtheorem{lemma}[theorem]{Lemma}

\newtheorem{remark}[theorem]{Remark}

\begin{document}

\title[Speed of Biased Walk on a Galton-Watson Tree without Leaves]{The Speed of a Biased Walk on a Galton-Watson Tree without Leaves is Monotonic
with Respect to Progeny Distributions for High Values of Bias}

\author{BEHZAD MEHRDAD}
\address
{Courant Institute of Mathematical Sciences\newline
\indent New York University\newline
\indent 251 Mercer Street\newline
\indent New York, NY-10012\newline
\indent United States of America}
\email{mehrdad@cims.nyu.edu
\\
sen@cims.nyu.edu
\\
ling@cims.nyu.edu}

\author{SANCHAYAN SEN}

\author{LINGJIONG ZHU}

\subjclass[2000]{60K37; 60J80; 60G50} 
\keywords{Random Walk in random environment, Galton-Watson tree, speed, stochastic domination.}

\begin{abstract}
Consider biased random walks on two Galton-Watson trees without leaves having
progeny distributions $P_1$ and $P_2$ (GW$(P_1)$ and GW$(P_2)$) where
$P_1$ and $P_2$ are supported on positive integers and $P_1$ dominates $P_2$ stochastically.
We prove that the speed of the walk on GW$(P_1)$ is bigger than the same on GW$(P_2)$ when
the bias is larger than a threshold depending on $P_1$ and $P_2$.
This partially answers a question raised by Ben Arous, Fribergh and
Sidoravicius.
\end{abstract}

\maketitle


\section{Introduction and Main Results}

\subsection{Introduction}

Consider a supercritical Galton-Watson tree, i.e. a random rooted tree, where the offspring size of
all individuals are i.i.d. copies of an integer random variable $Z$, which satisfies
$P(Z=k)=p_{k}$, $k=0,1,\ldots$. The tree has no leaves if $p_{0}=0$.
We will
use $|x|$ to denote the distance of a vertex $x$ from the root. Moreover $x_{\ast}$ will
denote the ancestor of $x$ for any vertex $x$ different from the root and $x_{i}$ will denote
the $i$th child of $x$. Given a random tree $T$ and $\beta>0$, we define $\beta$-biased random walk $(X_{n})_{n\geq 0}$
on $T$ as follows. Transitions to each of the children of
the root are equally likely. If the vertex $x$ has $k$ children and $x$ is not the root then the transition probabilities are given by
\begin{align*}
&P(X_{n+1}=x_{\ast}|X_{n}=x)=\frac{1}{1+\beta k},
\\
&P(X_{n+1}=x_{i}|X_{n}=x)=\frac{\beta}{1+\beta k},\quad i=1,2,\ldots,k.
\end{align*}
We start the walk from the root of the tree and denote by $P^{\omega}$ the law of
$(X_{n})_{n\geq 0}$ on a tree $\omega$.
We define the averaged law as the semi-direct product $\mathbb{P}=\overline P\times P^{\omega}$ where $\overline P$ is the Galton-Watson
measure (associated with offspring distribution $P$) on the space of rooted trees conditioned on non-extinction.

\cite{Lyons} proved that if $\beta>\frac{1}{E[Z]}$, then the random walk is transient, i.e.
$\lim_{n\rightarrow\infty}|X_{n}|=\infty$. \cite{LyonsII} showed that $\mathbb{P}$-almost surely the speed
\begin{equation}\label{defn of speed}
v(\beta,P):=\lim_{n\rightarrow\infty}\frac{|X_{n}|}{n}
\end{equation}
exists and is a non-random constant. A lot of work has been done
on the behavior of the speed as a function of $\beta$. \cite{LyonsII}
conjectured that $v(\beta,P)$ increases in $\beta$ on $(\frac{1}{E[Z]},\infty)$ when the tree has no leaves i.e. $P\{0\}=0$.
The conjecture has been open for a long time until proven recently in \cite{BenArous} for large values of $\beta$.


\begin{theorem*}[\cite{BenArous}]
The speed $v(\beta,P)$ of a $\beta$-biased random walk on a Galton-Watson tree without leaves is increasing for
$\beta>\beta_{c}$ for some $\beta_{c}>0$ very large when $P\{0\}=0$.
\end{theorem*}

Very recently, A\"{i}d\'{e}kon obtained an expression for the speed $v$.


\begin{theorem*}[\cite{Aidekon}]
\begin{equation}
v(\beta,P)=\frac{\mathbb{E}\left[\frac{(\beta Z-1)Y_{0}}{1-\beta+\beta\sum_{i=0}^{Z}Y_{i}}\right]}
{\mathbb{E}\left[\frac{(\beta Z+1)Y_{0}}{1-\beta+\beta\sum_{i=0}^{Z}Y_{i}}\right]},
\end{equation}
where $Y_{i}$ are i.i.d. copies distributed as $P_{x}(\tau_{x_{\ast}}=\infty)$, where $\tau_{y}$ is
the first hitting time of $y$.
\end{theorem*}

Using his own formula, A\"{i}d\'{e}kon (private communications) can prove the monotonicity for $\beta\geq2$
when $P\{0\}=0$.
However, the original conjecture is still open in the sense that it is not known if the monotonicity
holds for every $\beta>1/E[Z]$.

In this paper we shall investigate how the speed changes when one changes the progeny distribution
keeping the bias fixed.

The paper is organized as the following. In Section \ref{MainResults}, we will
introduce our main results. In Section \ref{CouplingSection}, we will describe in details
our coupling method. Finally, in Section \ref{ProofSection}, we will provide the proofs of all the
results in \ref{MainResults}.

\subsection{Main Results}\label{MainResults}

In \cite{BenArous}, the authors raised the following interesting question, if $P_{1}$ stochastically
dominates $P_{2}$, does it imply that $v(\beta,P_{1})\geq v(\beta,P_{2})$? We show that this is indeed the case at least
when the bias is large.

Throughout this paper, when we say $P_{1}$ stochastically dominates $P_{2}$, we also mean that $P_{1}\neq P_{2}$.
We also recall that if $P_1$ dominates $P_2$ stochastically then there is a coupling
of the random variables $Z_1$ and $Z_2$ having distributions $P_1$ and $P_2$ respectively
such that $Z_2\leq Z_1$.

We have the following result.


\begin{theorem}\label{MainThm}
Assume that $P_{1}$ and $P_{2}$ are two
probability measures on positive integers such that $P_{1}$ stochastically dominates $P_{2}$.
Consider $\beta$-biased random walks on $GW(P_{1})$ and $GW(P_{2})$. Then for every $\delta>0$ there
exists a $\beta_{0}:=\beta_{0}(P_{1},P_{2},\delta)>0$ such that
for any $\beta>\beta_{0}$,
we have $v(\beta,P_{1})>v(\beta,P_{2})$. The constant $\beta_{0}$ equals $\max\{\beta_1,\frac{23}{4}+\delta\}$ where
\begin{equation}\label{twolowerbounds}
\beta_{1}:= c_{\delta}\cdot\min\left\{\frac{E\left[\left(\frac{1}{Z_{1}}-\frac{1}{Z_{2}}\right)1_{Z_{1}<Z_{2}}\right]}
{E\left[\frac{1}{Z_{2}'}-\frac{1}{Z_{1}'}\right]},
\frac{E\left[Z_{2}'\left(\frac{1}{Z_{2}'}-\frac{1}{Z_{1}'}\right)\right]}
{E\left[\frac{1}{Z_{2}'}-\frac{1}{Z_{1}'}\right]}+1\right\},
\end{equation}
and $c_{\delta}$ is a universal constant depending only on $\delta$. Here, $Z_{1}$, $Z_{2}$ are independent and are distributed
according to $P_{1}$ and $P_{2}$ respectively, $Z_{1}'$ and $Z_{2}'$ are jointly distributed so that $Z_{1}'\geq Z_{2}'$ almost
surely and their marginal distributions are $P_{1}$ and $P_{2}$.
\end{theorem}

\begin{remark}
There is a universal cut-off $\beta_{1}=\beta_{1}(M)$ which works for all $P_{2}$ supported
on $\{1,2,\ldots,M\}$ since we have
\begin{equation}
E\left[Z_{2}'\left(\frac{1}{Z_{2}'}-\frac{1}{Z_{1}'}\right)\right]
\leq M\cdot E\left[\frac{1}{Z_{2}'}-\frac{1}{Z_{1}'}\right].
\end{equation}
The other expression inside the parentheses in the definition of $\beta_1$ in Theorem \ref{MainThm} is more useful when
``the distribution of $Z_{1}$ is much larger than that of $Z_{2}$''; we shall illustrate this in Corollary \ref{lowcutoff}.
\end{remark}

\begin{remark} Suppose $P_1$ dominates $P_2$ and are both supported on positive integers.
 Then $v(\beta,P_1)\geq v(\beta,P_2)$
follows trivially in the following cases.\\
(i) It is easy to see (via a coupling argument) that if the maximum of the support of $P_{2}$ is not larger
than the minimum of the support of $P_{1}$, then for any $\beta>0$, we have $v(\beta,P_{1})\geq v(\beta,P_{2})$.\\
(ii) We have $v(1,P_{1})\geq v(1,P_{2})$ just by considering the expression
$$v(1,P)=E_{P}\left[\frac{Z-1}{Z+1}\right]$$ obtained by \cite{LyonsI}.\\
(iii) Note that $v(1/E_{P_2}[Z],P_2)=0$,
$v(1/E_{P_2}[Z],P_1)>0$, and $v(\beta, P_j)$ is continuous in $\beta$ for $j=1,2$.
Thus, for some small $\epsilon >0$ we have $v(\beta,P_1)\geq v(\beta,P_2)$ for $0<\beta<\epsilon+1/E_{P_2}[Z]$.

Further (ii) and (iii) hold even when the offspring distributions are supported on non-negative
integers as long as we define the speed as in (\ref{defn of speed}) conditional on non-extinction
of the trees.
\end{remark}

We can improve the threshold $\beta_0$
of Theorem \ref{MainThm} by making stronger assumptions.


\begin{theorem}\label{KThm}
Suppose $P_{1}$ and $P_{2}$ are two probability measures on positive integers
such that for some $\ell>1$, there exists a coupling
of $Z^{(1)}_{1}, Z_{1}^{(2)},\cdots,Z_{1}^{(\ell)}$ and
$Z_{2}^{(1)}, Z_{2}^{(2)},\cdots, Z_{2}^{(\ell)}$ for which
$\min\{Z^{(1)}_{1}, Z_{1}^{(2)},\cdots, Z_{1}^{(\ell)}\}
\geq \max\{Z_{2}^{(1)}, Z_{2}^{(2)},\cdots, Z_{2}^{(\ell)}\}$ almost surely, where
$Z_{j}^{(1)},\ldots,Z_{j}^{(\ell)}$ are i.i.d. distributed according to $P_{j}$ for $j=1,\ 2$. Then
for any $\delta>0$, we have
$v(\beta,P_{1})\geq v(\beta,P_{2})$
for any $\beta>\max\{K\cdot\beta_{1}^{1/\ell}, \frac{23}{4}+\delta\}$ where the constant $K$
equals $\frac{27}{4}\cdot 3^{5/3}$.
\end{theorem}


\begin{corollary}\label{lowcutoff}
Assume that $P_{1}$ and $P_{2}$ are two probability measures on positive integers
such that $P_{1}$ stochastically dominates $P_{2}$. Let $m_{i}:=E_{P_i}[Z]$
and $Z_{i}^{(n)}$ be the number of children in the $n$th generation in $GW(P_{i})$, denote the law of $Z_{i}^{(n)}$
by $P_{i}^{(n)}$ for $i=1,\ 2$. Assume
that there exists some $\theta>0$ such that $E[e^{\theta Z_{1}^{(1)}}]<\infty$. Let $f$ be the generating function for $P_1$ and
$\alpha:=-\log f'(0)/\log f'(1)$. Further assume that $m_{1}>m_{2}^{\max\{\frac{2}{\alpha},\frac{1}{\alpha}+1\}}$
(if $P_1\{1\}=0$, then $\alpha=\infty$ and this condition is automatically satisfied).

Then, for any $\beta>23/4$, there exists some $k=k(P_{1},P_{2},\beta)$ such that $v(\beta,P_{1}^{(k)})>v(\beta,P_{2}^{(k)})$.
(We emphasize that $v(\beta,P_{i}^{(k)})$ is the speed of a $\beta$-biased random walk on a Galton-Watson tree
having $P_{i}^{(k)}$ as its offspring distribution.)
\end{corollary}

The following corollary is the counterpart for Theorem 1.2 in \cite{BenArous}.


\begin{corollary}\label{morekids}
Assume all the assumptions in Theorem \ref{MainThm}, moreover, assume that the minimum degrees of both $P_1$ and $P_2$ are bigger than $d$,
i.e.  $d_{i} := \min\{k \geq 1, P_{i}(Z = k) > 0\}\geq d$, for $i\in\{1,2\}$.

Then the result of Theorem \ref{MainThm} is true for a smaller $\beta$; that is  $v(\beta,P_{1})>v(\beta,P_{2})$
for any $\beta>\beta_{0}= \max\{\beta_{1}, \frac{23}{4d}+\delta\}$, where the constant  $c_{\delta}$
should be replaced with $\tilde{c}_{\delta}=\frac{c_{\delta}}{d}$
in the definition of $\beta_1$ in \eqref{twolowerbounds}.

\end{corollary}

Basically, the stronger assumption of graphs with degrees higher than $d$ enables us to strengthen our results
and $\beta$ can be reduced to $23/4d$ if $d$ is large enough.





\section{Constructing the Walks}\label{CouplingSection}

Let us describe precisely the coupling we use. Let $U_1$ have uniform distribution
on $(1/(\beta+1), 1)$. Let $(U_{i})_{i\geq 2}$ be
i.i.d. uniformly distributed random variables on $[0,1]$ independent of $U_1$. Let
$\{(Z_{1,k}',Z_{2,k}')\}_{k\geq 1}$ be i.i.d. random vectors
such that for each $k$, $Z_{1,k}'$ has the marginal distribution $P_{1}$ and $Z_{2,k}'$
has the marginal distribution $P_{2}$ and with probability $1$, we have $Z_{2,k}'\leq Z_{1,k}'$.
Finally let $\{Z_{i,k}\}_{k\geq 1}$ be i.i.d. $P_i$ for $i=1,2$. The sequences $\{U_i\}_{i\geq 1},\
\{Z_{1,k}\}_{k\geq 1},\ \{Z_{2,k}\}_{k\geq 1},\ \{(Z_{1,k}',Z_{2,k}')\}_{k\geq 1}$ are independent of each other.

In our proof we shall work conditional on an event which ensures that the roots are only visited once,
for this reason we only need one copy of $U_1$. Note that our definition of $U_1$ is slightly different
from the one in \cite{BenArous}.

We construct two random walks $X_n^{(1)}$ and $X_n^{(2)}$ (on $GW(P_1)$ and $GW(P_2)$)
and another walk $Y_n$ on $\mathbb{Z}_{\geq 0}$ in the following way.
Define $Y_{0}:=0$ and for $n\geq 1$,
\begin{equation}
Y_{n}:=\sum_{i=1}^{n}\left\{1_{U_{i}>\frac{1}{\beta+1}}-1_{U_{i}\leq\frac{1}{\beta+1}}\right\},\quad n\in\mathbb{N}.
\end{equation}

We start $X^{(1)}$ and $X^{(2)}$ at the roots and grow the trees $GW(P_1)$ and $GW(P_2)$ dynamically.
For simplicity we drop the time parameter $n$ and denote the position of $X_{n}^{(i)}$ by $x^{(i)}$.

Now, if at time $n\geq 0$,
$X_n^{(1)}$ and $X_n^{(2)}$ are at two sites $x^{(1)}$ and $x^{(2)}$, neither of them visited before by
the corresponding walks, then
we assign $Z_{1,n+1}'$ and $Z_{2,n+1}'$  many children to $x^{(1)}$ and $x^{(2)}$ respectively
(recall that $Z_{1,n+1}'\geq Z_{2,n+1}'$).

If at time $n$, one of the walks, say $X^{(1)}$ is at a site $x^{(1)}$ previously visited
by the walk while the other walk $X^{(2)}$  is at a new site $x^{(2)}$ then we assign
$Z_{2,n+1}$ many children to $x^{(2)}$.

Let us now explain the rules for transition.
Denote the number of offsprings of $x^{(i)}$ by $Z_{i}$ and
let $x^{(i)}_k$ be the $k$th child of $x^{(i)}$ ($i=1,\ 2$).

Define
\begin{align}
&\eta_{1}:=\frac{\beta}{(\beta+1)Z_{1}},\quad \eta_{2}:=\left(\frac{\beta}{\beta+1}\right)\left(\frac{1}{Z_2}-\frac{1}{Z_1}\right),
\quad \eta_{3}:=\left(\frac{1}{\beta+1}-\frac{1}{Z_{2}\beta+1}\right)\frac{1}{Z_{2}},
\\
&\eta_{4}:=\left(\frac{1}{\beta+1}-\frac{1}{Z_{1}\beta+1}\right)\frac{1}{Z_{1}},
\quad
\eta_{5}:=|\eta_{3}-\eta_{4}|.\nonumber
\end{align}

Then whenever $Z_1\geq Z_2$, we move according to the rule explained below.

When $U_{n+1}\in (1/(\beta+1),1)$ we have the following cases.
\begin{enumerate}
\item
Consider the random walk $X^{(1)}$.
\begin{itemize}
\item
If $U_{n+1}\in\big(\frac{1}{\beta+1}+(i-1)\eta_{1},\frac{1}{\beta+1}+i\eta_{1}\big]$,
then $X_{n+1}^{(1)}=x^{(1)}_{Z_1+1-i}$ for $i=1,2,\ldots,Z_{1}$.
\end{itemize}
\item
Consider the random walk $X^{(2)}$.
\begin{itemize}
\item
If $U_{n+1}\in\big(\frac{1}{\beta+1}+(i-1)\eta_{2},\frac{1}{\beta+1}+i\eta_{2}\big]$,
then we have $X_{n+1}^{(2)}=x^{(2)}_{Z_2+1-i}$, where $i=1,2,\ldots,Z_{2}$.
\item
If $U_{n+1}\in\big(\frac{1}{\beta+1}+Z_{2}\eta_{2}+(i-1)\eta_{1},
\frac{1}{\beta+1}+Z_{2}\eta_{2}+i\eta_{1}\big]$,
then we have $X_{n+1}^{(2)}=x^{(2)}_{Z_2+1-i}$, where $i=1,2,\ldots,Z_{2}$.
\end{itemize}
\end{enumerate}

\begin{figure}[htb]
\begin{center}
\includegraphics[scale=0.80]{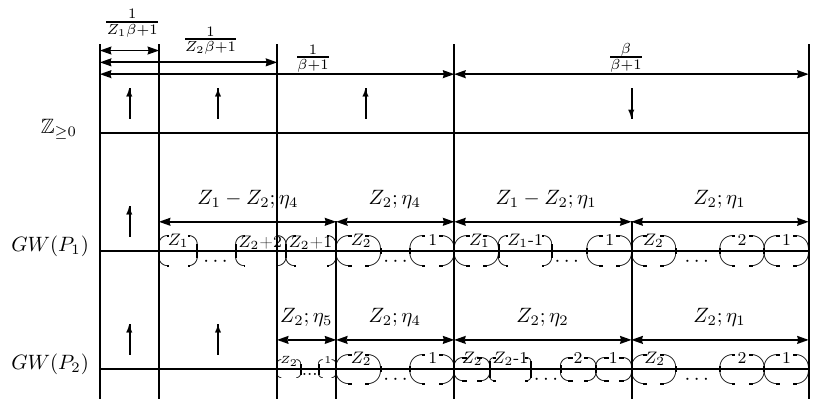}
\caption{The coupling for $\eta_{3}\geq\eta_{4}$. In the illustration, we use $Z_{1};\eta_{4}$ etc.
to denote $Z_{1}$ many subintervals with each subinterval of length $\eta_{4}$ etc.}
\label{PictureOne}
\end{center}
\end{figure}

When $U_{n+1}\in (0,1/(\beta+1))$ we have to consider two cases.
If $\eta_{3}\geq\eta_{4}$, then we use the following coupling. Figure \ref{PictureOne} gives an illustration.

\begin{enumerate}
\item
Consider the random walk $X^{(1)}$.
\begin{itemize}
\item
If $U_{n+1}\in\big[0,\frac{1}{Z_{1}\beta+1}\big]$, then we have $X_{n+1}^{(1)}=x^{(1)}_{\ast}$.
\item
If $U_{n+1}\in\big(\frac{1}{Z_{1}\beta+1}+(i-1)\eta_{4},\frac{1}{Z_{1}\beta+1}+i\eta_{4}\big]$,
then we have $X_{n+1}^{(1)}=x^{(1)}_{Z_1+1-i}$, where $i=1,2,\ldots,Z_{1}$.

\end{itemize}

\item
Consider the random walk $X^{(2)}$.
\begin{itemize}
\item
If $U_{n+1}\in\big[0,\frac{1}{Z_{2}\beta+1}\big]$, then we have $X_{n+1}^{(2)}=x^{(2)}_{\ast}$.
\item
If $U_{n+1}\in\big(\frac{1}{Z_{2}\beta+1}+(i-1)\eta_{5},\frac{1}{Z_{2}\beta+1}+i\eta_{5}\big]$,
then we have $X_{n+1}^{(2)}=x^{(2)}_{Z_2+1-i}$, where $i=1,2,\ldots,Z_{2}$.
\item
If $U_{n+1}\in\big(\frac{1}{Z_{2}\beta+1}+Z_{2}\eta_{5}+(i-1)\eta_{4},\frac{1}{Z_{2}\beta+1}+Z_{2}\eta_{5}+i\eta_{4}\big]$,
then we have $X_{n+1}^{(2)}=x^{(2)}_{Z_2+1-i}$, where $i=1,2,\ldots,Z_{2}$.
\end{itemize}
\end{enumerate}

\begin{figure}[htb]
\begin{center}
\includegraphics[scale=0.80]{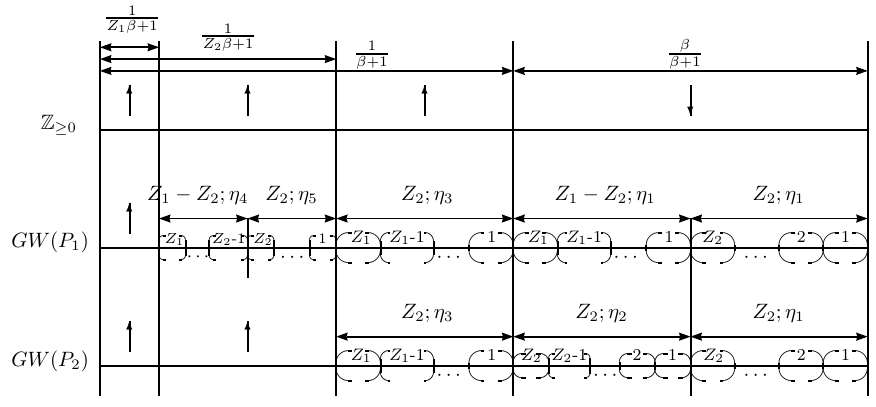}
\caption{The coupling for $\eta_{4}>\eta_{3}$}
\label{PictureTwo}
\end{center}
\end{figure}

If $\eta_{3}<\eta_{4}$, then we use the following coupling. Figure \ref{PictureTwo} is an illustration
of the following coupling.
\begin{enumerate}

\item
Consider the random walk $X^{(1)}$.
\begin{itemize}
\item
If $U_{n+1}\in\big[0,\frac{1}{Z_{1}\beta+1}\big]$, then we have $X_{n+1}^{(1)}=x^{(1)}_{\ast}$.
\item
If $U_{n+1}\in\big(\frac{1}{Z_{1}\beta+1}+(i-1)\eta_{4},\frac{1}{Z_{1}\beta+1}+i\eta_{4}\big]$,
then we have $X_{n+1}^{(1)}=x^{(1)}_{Z_1+1-i}$, where $i=1,2,\ldots,Z_{1}-Z_{2}$.
\item
If $U_{n+1}\in\big(\frac{1}{Z_{1}\beta+1}+(Z_{1}-Z_{2})\eta_{4}+(i-1)\eta_{5}
,\frac{1}{Z_{1}\beta+1}+(Z_{1}-Z_{2})\eta_{4}+i\eta_{5}\big]$,
then $X_{n+1}^{(1)}=x^{(1)}_{Z_2+1-i}$, where $i=1,2,\ldots,Z_{2}$.
\item
If $U_{n+1}\in\big(\frac{1}{Z_{2}\beta+1}+(i-1)\eta_{3},\frac{1}{Z_{2}\beta+1}+i\eta_{3}\big]$,
then we have $X_{n+1}^{(1)}=x^{(1)}_{Z_2+1-i}$, where $i=1,2,\ldots,Z_{2}$.
\end{itemize}

\item
Consider the random walk $X^{(2)}$.
\begin{itemize}
\item
If $U_{n+1}\in\big[0,\frac{1}{Z_{2}\beta+1}\big]$, then we have $X_{n+1}^{(2)}=x^{(2)}_{\ast}$.
\item
If $U_{n+1}\in\big(\frac{1}{Z_{2}\beta+1}+(i-1)\eta_{3},\frac{1}{Z_{2}\beta+1}+i\eta_{3}\big]$,
then we have $X_{n+1}^{(2)}=x^{(2)}_{Z_2+1-i}$, where $i=1,2,\ldots,Z_{2}$.
\end{itemize}
\end{enumerate}

Finally if $Z_1< Z_2$ we move according to the following rule.
\begin{enumerate}
\item
For $i=1,\ 2$
\begin{itemize}
\item
If $U_{n+1}\in\big[0,\frac{1}{Z_{i}\beta+1}\big]$, then we have $X_{n+1}^{(i)}=x^{(i)}_{\ast}$.
\item
If $U_{n+1}\in\big(\frac{1}{Z_{i}\beta+1}+(j-1)\frac{\beta}{Z_{i}\beta+1},\frac{1}{Z_{i}\beta+1}+j\frac{\beta}{Z_{i}\beta+1}]$,
then we have $X_{n+1}^{(i)}=x^{(i)}_{j}$, where $j=1,2,\ldots,Z_{i}$.
\end{itemize}
\end{enumerate}

It is routine to check that $X^{(i)}$ is a $\beta$-biased random walk on $GW(P_{i})$ for $i=1,\ 2$.

\section{Proofs}\label{ProofSection}

The main idea in our proof is to use a technique originally used in \cite{BenArous},
to couple the walks on the Galton-Watson trees with a random walk on $\mathbb{Z}$.
We will use a super-regeneration time which is a regeneration time for all the three walks $Y$, $GW(P_{1})$
and $GW(P_{2})$. Regeneration time is an often-used technique in the study of random walks in random media.
(See for example \cite{Zeitouni}.)
Informally, a regeneration time is a maximum of a random walk which is also a minimum of the future of the random walk.
A time $\tau$ is a regeneration time for the $\beta$-biased random walk $(Y_{n})_{n\geq 0}$ on $\mathbb{Z}$
if we have
\begin{equation}
Y_{\tau}>\max_{n<\tau}Y_{n}
\quad\text{and}\quad
Y_{\tau}<\min_{n>\tau}Y_{n}.
\end{equation}
Consider the regeneration time for walks on $GW(P_{1})$ and $GW(P_{2})$ in the sense that is usually defined
on trees (see \cite{LyonsII}). As in \cite{BenArous}, if $\tau$ is a regeneration time for $(Y_{n})_{n\geq 0}$, then it is also
a regeneration time for $GW(P_{1})$ and $GW(P_{2})$. In this respect, $\tau$ is called a super-regeneration time.

Let us consider the event that $0$ is a regeneration time for $(Y_{n})_{n\geq 0}$.
Following the notations in \cite{BenArous}, we denote this event by $\{0-SR\}$. Then, we have
\begin{equation}
p_{\infty}:=P(0-SR)=\frac{\beta-1}{\beta+1}.
\end{equation}
Let us define the probability measure $\tilde{P}$ as
\begin{equation}
\tilde{P}(\cdot):=P(\cdot|0-SR).
\end{equation}
Under $\tilde{P}$, $0$ is the first regeneration time and let $\tau_{i}$ be $i$th non-zero regeneration time.

Then,
$(|X_{\tau_{i+1}}-X_{\tau_{i}}|,\tau_{i+1}-\tau_{i})_{i\geq 1}$ is a sequence of i.i.d. random vectors
having the same distribution as $(|X_{\tau_{1}}|,\tau_{1})$ under $\tilde{P}$ and
as in \cite{BenArous},
we have, for any $\beta>1$,
\begin{equation}
v(\beta,P_{1})=\frac{\tilde{E}[|X_{\tau_{1}}^{(1)}|]}{\tilde{E}[\tau_{1}]}
\quad
\text{and}
\quad
v(\beta,P_{2})=\frac{\tilde{E}[|X_{\tau_{1}}^{(2)}|]}{\tilde{E}[\tau_{1}]}.
\end{equation}
Hence, $v(\beta,P_{1})>v(\beta,P_{2})$ is equivalent to $\tilde{E}[|X_{\tau_{1}}^{(1)}|]>\tilde{E}[|X_{\tau_{1}}^{(2)}|]$.

Follwing the notation in \cite{BenArous} let us denote by $\mathcal{B}$ the set of times
before $\tau_{1}$ when the random walk on $\mathbb{Z}_{\geq 0}$ takes
a step back, i.e. $\mathcal{B}=\{j\leq \tau_{1}|\ U_j\leq 1/(\beta+1)\}$.

We quote the following lemma from \cite{BenArous}.

\begin{lemma}[Lemma 4.1. \cite{BenArous}]\label{NandK}
If $\{|\mathcal{B}|=k\}$, then $\{\tau_{1}\leq 3k+2\}$.
\end{lemma}

\begin{proof}[Proof of Theorem \ref{MainThm}]

Consider $|\mathcal{B}|=k$, i.e. $\mathcal{B}=\{i_{1}<\cdots<i_{k}\}$, where $k\geq 1$ and $\tau_{1}=n$.
Let us make two simple observations.

(i) $|X_{\tau_{1}}^{(1)}|-|X_{\tau_{1}}^{(2)}|=2$ or $0$ when $k=1$.

(ii) $|X_{\tau_{1}}^{(1)}|-|X_{\tau_{1}}^{(2)}|\geq -2(k-1)$ when $k\geq 2$.

We have
\begin{align}
&\tilde{E}\left[|X_{\tau_{1}}^{(1)}|-|X_{\tau_{1}}^{(2)}|\right]
\\
&=\tilde{E}\left[|X_{\tau_{1}}^{(1)}|-|X_{\tau_{1}}^{(2)}|;|\mathcal{B}|=1\right]
+\sum\nolimits_{\ast}\tilde{E}\left[|X_{\tau_{1}}^{(1)}|-|X_{\tau_{1}}^{(2)}|;
\mathcal{B}=\{i_{1}<\cdots<i_{k}\},\tau_{1}=n\right]\nonumber
\\
&\geq\tilde{E}\left[|X_{\tau_{1}}^{(1)}|-|X_{\tau_{1}}^{(2)}|;|\mathcal{B}|=1\right]\nonumber
\\
&\ -\sum\nolimits_{\ast} 2(k-1)\tilde{P}\left(|X_{\tau_{1}}^{(1)}|-|X_{\tau_{1}}^{(2)}|<0;
\mathcal{B}=\{i_{1}<\cdots<i_{k}\},\tau_{1}=n\right)\nonumber
\end{align}
where $\sum_{\ast}$ stands for summation over all $n\geq 2,\ k\geq 2$ and $\{i_{1},\ldots,i_{k}\}\subseteq\{1,\ldots,n\}$
for which the walk $Y_k$ does not come back to the origin.

For the first term, we have
\begin{equation}\label{StepOne}
\tilde{E}\left[|X_{\tau_{1}}^{(1)}|-|X_{\tau_{1}}^{(2)}|;|\mathcal{B}|=1\right]
\geq 2\left(\frac{\beta}{\beta+1}\right)^{3}E\left[\frac{1}{Z_{2}\beta+1}-\frac{1}{Z_{1}\beta+1}\right].
\end{equation}
Note the small difference between \eqref{StepOne} and Lemma 5.1. in \cite{BenArous}, which is due
to the difference in the definition of $U_1$ as mentioned earlier.
Let us explain the inequality in \eqref{StepOne}. Let $\epsilon_i=\mathbb{I}(U_i\geq 1/(\beta+1))-\mathbb{I}(U_i<1/(\beta+1))$.
When $|\mathcal{B}|=1$, $|X_{\tau_{1}}^{(1)}|-|X_{\tau_{1}}^{(2)}|=2$ or $0$, hence we have
$$
\tilde{E}\left[|X_{\tau_{1}}^{(1)}|-|X_{\tau_{1}}^{(2)}|;|\mathcal{B}|=1\right]
=\frac{2}{p_{\infty}}P\left(|X_{\tau_{1}}^{(1)}|-|X_{\tau_{1}}^{(2)}|=2;|\mathcal{B}|=1;0-SR\right)$$
and thus we get the lower bound in \eqref{StepOne} by considering the event
$$\mathcal{A}=\{\epsilon_1=\epsilon_2=1,\ \epsilon_3=-1 \text{ and } |X^{(1)}_3|-|X^{(2)}_3|=2,\ \epsilon_4=\epsilon_5=1, \tau_1=5\}.$$

For the second term, we have
\begin{align}
&\tilde{P}\left(|X_{\tau_{1}}^{(1)}|-|X_{\tau_{1}}^{(2)}|<0;
\mathcal{B}=\{i_{1}<\cdots<i_{k}\},\tau_{1}=n\right)
\\
&\leq\frac{1}{p_{\infty}}P\left(|X_{\tau_{1}}^{(1)}|-|X_{\tau_{1}}^{(2)}|<0;
\mathcal{B}=\{i_{1}<\cdots<i_{k}\},\tau_{1}=n\right).\nonumber
\end{align}
On $\{|X_{\tau_{1}}^{(1)}|-|X_{\tau_{1}}^{(2)}|<0\}$, let $\sigma$ be the first time
when the walk on $GW(P_{1})$ goes up but the walk on $GW(P_{2})$ goes down, necessarily $\sigma\in\mathcal{B}$. We introduce some
notation here, given a sequence $\theta=\{\theta_n\}_{n\geq 1}$ where $\theta_n=\pm 1$ we denote by
$\tau(\theta)$ the first regeneration time for the walk $Z_n=\sum_{i=1}^n \theta_i$, e.g.
$\tau_1=\tau(\epsilon)$ where $\epsilon=\{\epsilon_n=\mathbb{I}(U_n\geq 1/(\beta+1))-\mathbb{I}(U_n<1/(\beta+1)\}_{n\geq 1}$.
Define
$$\tau_1^{(j)}=\tau(\epsilon^{(j)})\text{ where }\epsilon^{(j)}=\{\epsilon_1,\hdots,\epsilon_{j-1},-1,\epsilon_{j+1},\hdots\}.$$
We can define $\mathcal{B}^{(j)}$ similarly. Also define
$$\tau_{1_{(j)}}=\tau(\epsilon_{(j)})\text{ where }\epsilon_{(j)}=\{\epsilon_1,\hdots,\epsilon_{j-1},+1,\epsilon_{j+1},\hdots\}.$$
Define $\mathcal{B}_{(j)}$ in an analogous manner. Also note that
if $|X_{\tau_{1}}^{(1)}|-|X_{\tau_{1}}^{(2)}|<0$ then the event
$$\mathcal{E}=\bigcup_{i,j\leq \tau_1}\{Z_{1,i}<Z_{2,j}'\}
              \bigcup_{i,j\leq \tau_1}\{Z_{1,i}'<Z_{2,j}\}
              \bigcup_{i,j\leq \tau_1}\{Z_{1,i}<Z_{2,j}\}
              \bigcup_{\substack{i\neq j\\ i,j\leq \tau_1}}\{Z_{1,i}'<Z_{2,j}'\}$$
is true. Let $\mathcal{E}_{i_l}:=\bigcup_{j=1}^4 \mathcal{E}_{j,i_l}$ where
$$\mathcal{E}_{1,i_l}:=\bigcup_{i,j\leq \tau_1}\left[\{Z_{1,i}<Z_{2,j}'\}
\bigcap\left\{U_{i_l}\in\left(\frac{1}{Z_{2,j}'\beta+1},\frac{1}{Z_{1,i}\beta+1}\right)\right\}\right]$$
and the other three events are defined similarly.
\begin{align}
&P\left(|X_{\tau_{1}}^{(1)}|-|X_{\tau_{1}}^{(2)}|<0;
\mathcal{B}=\{i_{1}<\cdots<i_{k}\},\tau_{1}=n\right)\nonumber
\\
&\leq\sum_{\ell=1}^{k}P\left(
\mathcal{B}=\{i_{1}<\cdots<i_{k}\},\tau_{1}=n,\sigma=i_{\ell}\right)\nonumber
\\
&=\sum_{\ell=1}^{k}P\left(\mathcal{B}^{(i_{\ell})}=\{i_{1}<\cdots<i_{k}\},\tau_{1}^{(i_{\ell})}=n,\sigma=i_{\ell}\right)\nonumber
\\
&\leq\sum_{\ell=1}^{k}P\bigg(\left\{\mathcal{B}^{(i_{\ell})}=\{i_{1}<\cdots<i_{k}\},\tau_{1}^{(i_{\ell})}=n\right\};
\mathcal{E}_{i_l}
\bigg)\nonumber
\\
&\leq\sum_{\ell=1}^{k}4n^{2}P\left(\mathcal{B}^{(i_{\ell})}=\{i_{1}<\cdots<i_{k}\},\tau_{1}^{(i_{\ell})}=n\right)
E\left[\frac{1}{Z_{1}\beta+1}-\frac{1}{Z_{2}\beta+1};1_{Z_{1}<Z_{2}}\right],\nonumber
\end{align}
where we used independence of $\epsilon^{(i_{\ell})}$ and $U_{i_{\ell}}$. Then, by Lemma \ref{NandK},
\begin{align}
&P\left(|X_{\tau_{1}}^{(1)}|-|X_{\tau_{1}}^{(2)}|<0;
\mathcal{B}=\{i_{1}<\cdots<i_{k}\},\tau_{1}=n\right)\nonumber
\\
&\leq4(3k+2)^{2}\sum_{\ell=1}^{k}P\left(\mathcal{B}^{(i_{\ell})}=\{i_{1}<\cdots<i_{k}\},
\tau_{1}^{(i_{\ell})}=n,U_{i_{\ell}}\leq\frac{1}{\beta+1}\right)\nonumber
\\
&\hskip70pt\cdot(\beta+1)\cdot
E\left[\frac{1}{Z_{1}\beta+1}-\frac{1}{Z_{2}\beta+1};1_{Z_{1}<Z_{2}}\right]\nonumber
\\
&=4(\beta+1)(3k+2)^{2}E\left[\frac{1}{Z_{1}\beta+1}-\frac{1}{Z_{2}\beta+1};1_{Z_{1}<Z_{2}}\right]
\sum_{\ell=1}^{k}P\left(\mathcal{B}=\{i_{1}<\cdots<i_{k}\},\tau_{1}=n\right)\nonumber
\\
&\leq 8\beta k(3k+2)^{2}E\left[\frac{(Z_{2}-Z_{1})\beta}{(Z_{1}\beta+1)(Z_{2}\beta+1)}1_{Z_{1}<Z_{2}}\right]
P\left(\mathcal{B}=\{i_{1}<\cdots<i_{k}\},\tau_{1}=n\right)\nonumber
\\
&\leq 8k(3k+2)^{2}E\left[\left(\frac{1}{Z_{1}}-\frac{1}{Z_{2}}\right)1_{Z_{1}<Z_{2}}\right]
P\left(\mathcal{B}=\{i_{1}<\cdots<i_{k}\},\tau_{1}=n\right).\nonumber
\end{align}
Therefore, by using the simple upper bound $P(|\mathcal{B}|=k)\leq c\left(\frac{27}{4(1+\beta)}\right)^{k}$
(Lemma 6.1. in \cite{BenArous}) for a universal constant $c$
and the fact that $p_{\infty}=(\beta-1)/(\beta+1)$,
we get
\begin{align}\label{finishing}
&\tilde{E}\left[|X_{\tau_{1}}^{(1)}|-|X_{\tau_{1}}^{(2)}|\right]
\\
&\geq 2\left(\frac{\beta}{\beta+1}\right)^{3}E\left[\frac{1}{Z_{2}\beta+1}-\frac{1}{Z_{1}\beta+1}\right]
\nonumber
\\
&\qquad-\sum_{k=2}^{\infty}\frac{16}{p_{\infty}}k(k-1)(3k+2)^{2}E\left[\left(\frac{1}{Z_{1}}-\frac{1}{Z_{2}}\right)1_{Z_{1}<Z_{2}}\right]
P(|\mathcal{B}|=k)\nonumber
\\
&\geq 2\left(\frac{\beta}{\beta+1}\right)^{3}E\left[\frac{(Z_{1}'-Z_{2}')\beta}{(Z_{2}'\beta+1)(Z_{1}'\beta+1)}\right]\nonumber
\\
&\qquad-\frac{c}{p_{\infty}}E\left[\left(\frac{1}{Z_{1}}-\frac{1}{Z_{2}}\right)1_{Z_{1}<Z_{2}}\right]\sum_{k=2}^{\infty}
16k(k-1)(3k+2)^{2}\left(\frac{27}{4(1+\beta)}\right)^{k}
\nonumber
\\
&\geq 2\left(\frac{\beta}{\beta+1}\right)^{3}E\left[\frac{(Z_{1}'-Z_{2}')\beta}{4Z_{2}'Z_{1}'\beta^{2}}\right]\nonumber
\\
&\qquad-\frac{c(\beta+1)}{(\beta-1)}E\left[\left(\frac{1}{Z_{1}}-\frac{1}{Z_{2}}\right)1_{Z_{1}<Z_{2}}\right]\sum_{k=2}^{\infty}
16k(k-1)(3k+2)^{2}\left(\frac{27}{4(1+\beta)}\right)^{k}\nonumber
\\
&\geq 2\left(\frac{\beta}{\beta+1}\right)^{3}E\left[\frac{(Z_{1}'-Z_{2}')}{4Z_{2}'Z_{1}'\beta}\right]
-\frac{c\cdot 27^2}{4^2(\beta-1)(\beta+1)}
E\left[\left(\frac{1}{Z_{1}}-\frac{1}{Z_{2}}\right)1_{Z_{1}<Z_{2}}\right]\nonumber
\\
&\hskip140pt\cdot\sum_{k=2}^{\infty}
16k(k-1)(3k+2)^{2}\left(\frac{27}{4(1+\beta)}\right)^{k-2},\nonumber
\end{align}
where we used the fact that $Z_{1}'\geq Z_{2}'\geq 1$ and $\beta>1$.
Hence we conclude that for any $\delta>0$, $\tilde{E}\left[|X_{\tau_{1}}^{(1)}|-|X_{\tau_{1}}^{(2)}|\right]>0$
if we have
\begin{equation}\label{maxoftwo}
\beta>\max\left\{c_{\delta}\cdot\frac{E\left[\left(\frac{1}{Z_{1}}-\frac{1}{Z_{2}}\right)1_{Z_{1}<Z_{2}}\right]}
{E\left[\frac{1}{Z_{2}'}-\frac{1}{Z_{1}'}\right]},\frac{23}{4}+\delta\right\},
\end{equation}
for some universal constant $c_{\delta}>0$ that only depends on $\delta>0$.

Now we derive the other lower bound in \eqref{twolowerbounds}. On $\{|X_{\tau_{1}}^{(1)}|-|X_{\tau_{1}}^{(2)}|<0\}$, let us define
the events $E$ and $F$ as
\begin{equation}
E:=\left\{\text{For some $\sigma_{1}\leq\tau_{1}$, $|X_{\sigma_{1}+1}^{(1)}|\neq|X_{\sigma_{1}+1}^{(2)}|$
and $X_{j}^{(1)}=X_{j}^{(2)}$ for any $j\leq\sigma_{1}$}\right\}.
\end{equation}
\begin{align}
F:=\big\{&\text{For some $\sigma_{2}\leq\tau_{1}$, $X_{j}^{(1)}=X_{j}^{(2)}$ for any $j\leq\sigma_{2}$},
\\
&\text{and $X_{\sigma_{2}+1}^{(1)}\neq X_{\sigma_{2}+1}^{(2)}$, but $|X_{\sigma_{2}+1}^{(1)}|=|X_{\sigma_{2}+1}^{(2)}|$}\big\}.\nonumber
\end{align}
In other words, $E$ is the event that the first time the walks on $GW(P_{1})$ and $GW(P_{2})$ decouple,
the walk on $GW(P_{2})$ goes up and the walk on $GW(P_{1})$ goes down. Clearly this happens at time $\sigma_{1}\in\mathcal{B}$.
$F$ is the event that the first time the walks on $GW(P_{1})$ and $GW(P_{2})$ decouple, they both go downwards but to different
offsprings. This happens at time $\sigma_{2}$ which may or may not be in $\mathcal{B}$.

Next,
\begin{align}\label{TwoTerms}
&P(|X_{\tau_{1}}^{(1)}|-|X_{\tau_{1}}^{(2)}|<0;\mathcal{B}=\{i_{1},\ldots,i_{k}\},\tau_{1}=n)
\\
&=P(|X_{\tau_{1}}^{(1)}|-|X_{\tau_{1}}^{(2)}|<0;\mathcal{B}=\{i_{1},\ldots,i_{k}\},\tau_{1}=n;E)\nonumber
\\
&\qquad\qquad\qquad+P(|X_{\tau_{1}}^{(1)}|-|X_{\tau_{1}}^{(2)}|<0;\mathcal{B}=\{i_{1},\ldots,i_{k}\},\tau_{1}=n;F).\nonumber
\end{align}
Let us get an upper bound for the second term in \eqref{TwoTerms}.
\begin{align}\label{SecondTerm}
&P(|X_{\tau_{1}}^{(1)}|-|X_{\tau_{1}}^{(2)}|<0;\mathcal{B}=\{i_{1},\ldots,i_{k}\},\tau_{1}=n;F)
\\
&=\sum_{\ell=1}^{n}P(|X_{\tau_{1}}^{(1)}|-|X_{\tau_{1}}^{(2)}|<0;\mathcal{B}=\{i_{1},\ldots,i_{k}\},\tau_{1}=n,\sigma_{2}=\ell;F)\nonumber
\\
&\leq\sum_{\ell=1}^{n}P(\mathcal{B}=\{i_{1},\ldots,i_{k}\},\tau_{1}=n,\sigma_{2}=\ell;F).\nonumber
\end{align}
If $\ell\notin\{i_{1},\ldots,i_{k}\}$, then we get
\begin{align}\label{ETerm}
& P(\mathcal{B}=\{i_{1},\ldots,i_{k}\},\tau_{1}=n,\sigma_{2}=\ell;F)
\\
&\leq P\left(\mathcal{B}=\{i_{1},\ldots,i_{k}\},\tau_{1}=n;
U_{\ell}\in\bigcup_{m=1}^{n}\left(\frac{1}{\beta+1},\frac{(Z_{1,m}'-Z_{2,m}')\beta}{(\beta+1)Z_{1,m}'}+\frac{1}{\beta+1}\right)\right)\nonumber
\\
&= P\left(\mathcal{B}_{(\ell)}=\{i_{1},\ldots,i_{k}\},\tau_{1(\ell)}=n,
U_{\ell}\in\bigcup_{m=1}^{n}\left(\frac{1}{\beta+1},\frac{(Z_{1,m}'-Z_{2,m}')\beta}{(\beta+1)Z_{1,m}'}+\frac{1}{\beta+1}\right)\right)\nonumber
\\
&= P\left(\mathcal{B}_{(\ell)}=\{i_{1},\ldots,i_{k}\},\tau_{1(\ell)}=n\right)
P\left(U_{\ell}\in\bigcup_{m=1}^{n}\left(\frac{1}{\beta+1},\frac{(Z_{1,m}'-Z_{2,m}')\beta}{(\beta+1)Z_{1,m}'}+\frac{1}{\beta+1}\right)\right)\nonumber
\\
&\leq n\frac{(\beta+1)}{\beta}P\left(\mathcal{B}_{(\ell)}=\{i_{1},\ldots,i_{k}\},\tau_{1(\ell)}=n,U_{\ell}\geq\frac{1}{\beta+1}\right)
\cdot E\left[1-\frac{Z_{2}'}{Z_{1}'}\right]\nonumber
\\
&\leq 2n P\left(\mathcal{B}=\{i_{1},\ldots,i_{k}\},\tau_{1}=n\right)\cdot E\left[1-\frac{Z_{2}'}{Z_{1}'}\right].\nonumber
\end{align}
If $\ell\in\{i_{1},\ldots,i_{k}\}$, let us define
\begin{equation}
G_{m}:=Z_{2,m}'\left[\left(\frac{1}{\beta+1}
-\frac{1}{Z_{2,m}'\beta+1}\right)\frac{1}{Z_{2,m}'}-\left(\frac{1}{\beta+1}
-\frac{1}{Z_{1,m}'\beta+1}\right)\frac{1}{Z_{1,m}'}\right]_{+}.
\end{equation}
Then, we get
\begin{align}
& P(\mathcal{B}=\{i_{1},\ldots,i_{k}\},\tau_{1}=n,\sigma_{2}=\ell;F)\nonumber
\\
&\leq P\left(\mathcal{B}=\{i_{1},\ldots,i_{k}\},\tau_{1}=n;
U_{\ell}\in\bigcup_{m=1}^{n}\left(\frac{1}{\beta Z_{2,m}'+1},\frac{1}{\beta Z_{2,m}'+1}
+G_{m}\right)\right)\nonumber
\\
&\leq P\left(\mathcal{B}^{(\ell)}=\{i_{1},\ldots,i_{k}\},\tau_1^{(\ell)}=n\right)
\cdot n\cdot E[G_{m}]\nonumber
\\
&=(\beta+1)P\left(\mathcal{B}^{(\ell)}=\{i_{1},\ldots,i_{k}\},\tau_1^{(\ell)}=n,U_{\ell}\leq\frac{1}{\beta+1}\right)
\cdot n\cdot E[G_{m}]\nonumber
\\
&=(\beta+1)P\left(\mathcal{B}=\{i_{1},\ldots,i_{k}\},\tau_{1}=n\right)\cdot n\cdot E[G_{m}].\nonumber
\end{align}
For a coupled $(Z_{1}',Z_{2}')$, and after a little bit of computations, we have,
\begin{align}
&\left(\frac{1}{\beta+1}-\frac{1}{Z_{2}'\beta+1}\right)-Z_{2}'\left(\frac{\beta}{Z_{1}'\beta+1}-\frac{\beta}{(\beta+1)Z_{1}'}\right)\nonumber
\\
&=\left(\frac{1}{\beta+1}\right)\left[1-\frac{(Z_{1}'-1)\beta}{Z_{1}'\beta+1}-\frac{\beta+1}{Z_{2}'\beta+1}
+\left(\frac{(Z_{1}'-1)\beta}{Z_{1}'\beta+1}\right)\left(1-\frac{Z_{2}'}{Z_{1}'}\right)\right].\nonumber
\end{align}
It is easy to check that
\begin{equation}
1-\frac{(Z_{1}'-1)\beta}{Z_{1}'\beta+1}-\frac{\beta+1}{Z_{2}'\beta+1}
=\frac{(Z_{2}'-Z_{1}')\beta+(Z_{2}'-Z_{1}')\beta^{2}}{(Z_{1}'\beta+1)(Z_{2}'\beta+1)}\leq 0,\nonumber
\end{equation}
and
\begin{equation}
0\leq\frac{(Z_{1}'-1)\beta}{Z_{1}'\beta+1}\left(1-\frac{Z_{2}'}{Z_{1}'}\right)\leq 1-\frac{Z_{2}'}{Z_{1}'}.\nonumber
\end{equation}
Hence $E[G_{m}]\leq\left(\frac{1}{\beta+1}\right)E\left[\left(1-\frac{Z_{2}'}{Z_{1}'}\right)\right]$ and therefore
\begin{align}\label{GTerm}
&(\beta+1)P\left(\mathcal{B}=\{i_{1},\ldots,i_{k}\},\tau_{1}=n\right)\cdot n\cdot E[G_{m}]
\\
&\leq(\beta+1)P\left(\mathcal{B}=\{i_{1},\ldots,i_{k}\},\tau_{1}=n\right)\cdot n\cdot\frac{1}{\beta+1}E\left[1-\frac{Z_{2}'}{Z_{1}'}\right]\nonumber
\\
&=nP\left(\mathcal{B}=\{i_{1},\ldots,i_{k}\},\tau_{1}=n\right)E\left[1-\frac{Z_{2}'}{Z_{1}'}\right].\nonumber
\end{align}
So plugging \eqref{ETerm} and \eqref{GTerm} back into \eqref{SecondTerm}, we get
\begin{align}
&P\left(|X_{\tau_{1}}^{(1)}|-|X_{\tau_{1}}^{(2)}|<0;\mathcal{B}=\{i_{1},\ldots,i_{k}\},\tau_{1}=n;F\right)
\\
&\leq 2n^{2}P\left(\mathcal{B}=\{i_{1},\ldots,i_{k}\},\tau_{1}=n\right)E\left[1-\frac{Z_{2}'}{Z_{1}'}\right]\nonumber
\\
&\leq 2(3k+2)^{2}P\left(\mathcal{B}=\{i_{1},\ldots,i_{k}\},\tau_{1}=n\right)E\left[1-\frac{Z_{2}'}{Z_{1}'}\right].\nonumber
\end{align}
This takes care of the second term in \eqref{TwoTerms}.
Finally, let us give an upper bound for the first term in \eqref{TwoTerms}.
We omit some of the steps since they are similar. In the following
computations, remember that $\sigma_{1}\in\mathcal{B}$.
\begin{align}
&P\left(|X_{\tau_{1}}^{(1)}|-|X_{\tau_{1}}^{(2)}|<0;\mathcal{B}=\{i_{1},\ldots,i_{k}\},\tau_{1}=n;E\right)
\\
&\leq\sum_{m=1}^{k}P\left(\mathcal{B}=\{i_{1},\ldots,i_{k}\},\tau_{1}=n;\sigma_{1}=i_{m};E\right)\nonumber
\\
&\leq knP\left(\mathcal{B}=\{i_{1},\ldots,i_{k}\},\tau_{1}=n\right)
\cdot(\beta+1)\cdot E\left[\frac{1}{Z_{2}'\beta+1}-\frac{1}{Z_{1}'\beta+1}\right]\nonumber
\\
&\leq k(3k+2)P\left(\mathcal{B}=\{i_{1},\ldots,i_{k}\},\tau_{1}=n\right)\cdot(\beta+1)\cdot
E\left[\frac{\beta(Z_{1}'-Z_{2}')}{Z_{1}'Z_{2}'\beta^{2}}\right]\nonumber
\\
&\leq 2k(3k+2)P\left(\mathcal{B}=\{i_{1},\ldots,i_{k}\},\tau_{1}=n\right)\cdot E\left[\frac{1}{Z_{2}'}-\frac{1}{Z_{1}'}\right].\nonumber
\end{align}
Similar to our arguments in \eqref{finishing}, we get
\begin{align}\label{finishingI}
&\tilde{E}\left[|X_{\tau_{1}}^{(1)}|-|X_{\tau_{1}}^{(2)}|\right]
\\
&\geq 2\left(\frac{\beta}{\beta+1}\right)^{3}E\left[\frac{1}{Z_{2}'\beta+1}-\frac{1}{Z_{1}'\beta+1}\right]
\nonumber
\\
&\qquad\qquad
-\frac{1}{p_{\infty}}\sum_{k=2}^{\infty}2(3k+2)^{2}P(|\mathcal{B}|=k)E\left[1-\frac{Z_{2}'}{Z_{1}'}\right]\nonumber
\\
&\qquad\qquad\qquad\qquad
-\frac{1}{p_{\infty}}\sum_{k=2}^{\infty}2k(3k+2)P(|\mathcal{B}|=k)\cdot E\left[\frac{1}{Z_{2}'}-\frac{1}{Z_{1}'}\right]\nonumber
\end{align}
\begin{align*}
&\geq
\left(\frac{1}{2}\right)\left(\frac{\beta}{\beta+1}\right)^{3}
E\left[\frac{1}{Z_{2}'}-\frac{1}{Z_{1}'}\right]
\\
&\qquad\qquad
-\frac{c}{p_{\infty}}E\left[1-\frac{Z_{2}'}{Z_{1}'}\right]\sum_{k=2}^{\infty}2(3k+2)^{2}\left(\frac{27}{4(1+\beta)}\right)^{k}
\\
&\qquad\qquad\qquad\qquad
-\frac{c}{p_{\infty}}E\left[\frac{1}{Z_{2}'}-\frac{1}{Z_{1}'}\right]\sum_{k=2}^{\infty}2k(3k+2)\left(\frac{27}{4(1+\beta)}\right)^{k}.
\end{align*}
As earlier, we conclude that for any $\delta>0$ there is a universal constant $c_{\delta}'$ such that
\begin{equation}
\tilde{E}\left[|X_{\tau_{1}}^{(1)}|-|X_{\tau_{1}}^{(2)}|\right]>0,
\end{equation}
whenever
\begin{equation}
\beta>\max\left\{c_{\delta}'\left(
\frac{E\left[Z_{2}'\left(\frac{1}{Z_{2}'}-\frac{1}{Z_{1}'}\right)\right]}
{E\left[\frac{1}{Z_{2}}-\frac{1}{Z_{1}}\right]}+1\right),\frac{23}{4}+\delta\right\}.
\end{equation}
\end{proof}

\begin{proof}[Proof of Corollary \ref{lowcutoff}]
We shall write $Z_i$ for $Z_i^{(1)}$, for $i=1,\ 2$ and $p_j$ for $P_1\{j\}$.
Let us first prove that $E\left[m_{2}^{k}/Z_{1}^{(k)}\right]\rightarrow 0$ as $k\rightarrow\infty$.
Pick up some $m_{3}$ satisfies $m_{2}<m_{3}<m_{1}$. Then, we have
\begin{align}
m_{2}^{k}E\left[\frac{1}{Z_{1}^{(k)}}\right]
&=m_{2}^{k}\sum_{n\leq m_{3}^{k}}\frac{1}{n}P(Z_{1}^{(k)}=n)
+m_{2}^{k}\sum_{n>m_{3}^{k}}\frac{1}{n}P(Z_{1}^{(k)}=n)\nonumber
\\
&\leq m_{2}^{k}P(Z_{1}^{(k)}\leq m_{3}^{k})+\frac{m_{2}^{k}}{m_{3}^{k}}.\nonumber
\end{align}
Therefore, it is sufficient to prove that $m_{2}^{k}P(Z_{1}^{(k)}\leq m_{3}^{k})\rightarrow 0$
as $k\rightarrow\infty$.

If $W_i$ denotes the almost sure limit of the martingale
$Z_{i}^{(k)}/m_{i}^{k}$, then under the assumption $E[Z_{i}\log^+ Z_{i}]<\infty$, $W_i$ is a
positive random variable for $i=1,\ 2$ (see e.g. \cite{Kesten} and \cite{LyonsIII}).
Several other properties of $W_i$ have been well studied in the literature.
Recall that $f$ is the generating function of $Z_{1}$, then $0<\alpha=-\log f'(0)/\log f'(1)$.
Let us first consider the case $p_{1}>0$.
Note that $\alpha<\infty$ when $p_{1}>0$.
From \cite{Bingham} and the references therein, if $p_{1}>0$, then, there exists a positive constant $D$
such that $P(W_1\leq\epsilon)\leq D\epsilon^{\alpha}$ as $\epsilon\downarrow 0$.

Moreover, \cite{Athreya} proved that
if there exists some $\theta>0$ such that $E[e^{\theta Z_{1}}]<\infty$ and $p_{j}\neq 1$ for
any $j\geq 1$, then there exist some constants $C_{1},C_{2}$
such that
\begin{equation}
P\left(\left|\frac{Z_{1}^{(k)}}{m_{1}^{k}}-W_1\right|\geq\epsilon\right)\leq C_{1}e^{-C_{2}\epsilon^{\frac{2}{3}}m_{1}^{\frac{k}{3}}}.
\end{equation}
Now, splitting $P(Z_{1}^{(k)}\leq m_{3}^{k})$ into two terms, we get
\begin{align}
P(Z_{1}^{(k)}\leq m_{3}^{k})
&=P\left(Z_{1}^{(k)}\leq m_{3}^{k},\left|\frac{Z_{1}^{(k)}}{m_{1}^{k}}-W_1\right|>\epsilon^{(k)}\right)
\\
&\qquad\qquad\qquad\qquad\qquad
+P\left(Z_{1}^{(k)}\leq m_{3}^{k},\left|\frac{Z_{1}^{(k)}}{m_{1}^{k}}-W_1\right|\leq\epsilon^{(k)}\right)\nonumber
\\
&\leq P\left(\left|\frac{Z_{1}^{(k)}}{m_{1}^{k}}-W_1\right|>\epsilon^{(k)}\right)
+P\left(W_1\leq\epsilon^{(k)}+\frac{m_{3}^{k}}{m_{1}^{k}}\right).\nonumber
\end{align}
Let us choose $\epsilon^{(k)}=m_{2}^{-\frac{k}{\alpha}-k\delta}$ for some $\delta>0$.

Using the results stated before from \cite{Bingham},
\begin{align}
m_{2}^{k}P\left(W_1\leq\epsilon^{(k)}+\frac{m_{3}^{k}}{m_{1}^{k}}\right)
&\leq Dm_{2}^{k}\left(\epsilon^{(k)}+\frac{m_{3}^{k}}{m_{1}^{k}}\right)^{\alpha}
\\
&=D\left(m_{2}^{-k\delta}+\left(\frac{m_{2}^{\frac{1}{\alpha}}m_{3}}{m_{1}}\right)^{k}\right)^{\alpha}\rightarrow 0,\nonumber
\end{align}
as $k\rightarrow\infty$ if we have $m_{1}>m_{2}^{\frac{1}{\alpha}}m_{3}$. Since
it is valid for any $m_{2}<m_{3}<m_{1}$, the condition $m_{1}>m_{2}^{\frac{1}{\alpha}+1}$ is enough.

Using the results stated before from \cite{Athreya},
\begin{equation}
m_{2}^{k}P\left(\left|\frac{Z_{1}^{(k)}}{m_{1}^{k}}-W_1\right|>\epsilon^{(k)}\right)
\leq m_{2}^{k}C_{1}e^{-C_{2}(\epsilon^{(k)})^{\frac{2}{3}}m_{1}^{\frac{k}{3}}}
=m_{2}^{k}C_{1}e^{-C_{2}m_{2}^{-\frac{2k}{3\alpha}-\frac{2}{3}k\delta}m_{1}^{\frac{k}{3}}}\rightarrow 0,
\end{equation}
as $k\rightarrow\infty$ if $m_{1}>m_{2}^{\frac{2}{\alpha}+2\delta}$. Since we can pick up any $\delta>0$,
the condition $m_{1}>m_{2}^{\frac{2}{\alpha}}$ is enough.
This proves that $E\left[m_{2}^{k}/Z_{1}^{(k)}\right]\rightarrow 0$ as $k\rightarrow\infty$.

If $p_{1}=0$, then $\kappa:=\min\{k>0:p_{k}>0\}\geq 2$ and from \cite{Bingham}, we have
$\log P(W_{1}\leq\epsilon)\leq -C\epsilon^{-\beta/(1-\beta)}$, for some positive constant $C$ and $\beta:=\log\kappa/\log m_{1}$.
In other words, $P(W_{1}\leq\epsilon)$ is exponentially small.
Since $m_{1}>m_{2}$, we can pick up some $\alpha'$ large enough such that $m_{1}>m_{2}^{\max\{\frac{2}{\alpha'},\frac{1}{\alpha'}+1\}}$
holds. Since $P(W_{1}\leq\epsilon)$ is exponentially small, we can find a positive constant $D'$
such that $P(W_{1}\leq\epsilon)\leq D'\epsilon^{\alpha'}$.
Repeat the arguments as in the case $p_{1}>0$ replacing $\alpha$ by $\alpha'$ and $D$ by $D'$.
This proves that $E\left[m_{2}^{k}/Z_{1}^{(k)}\right]\rightarrow 0$ as $k\rightarrow\infty$.

Now, let us go back to the proof of the corollary. From \eqref{twolowerbounds}, it suffices to show that
\begin{equation}\label{gotozero}
\frac{E\left[\left(\frac{1}{Z_{1}^{(k)}}-\frac{1}{Z_{2}^{(k)}}\right)1_{Z_{1}^{(k)}<Z_{2}^{(k)}}\right]}
{E\left[\frac{1}{Z_{2}^{'(k)}}-\frac{1}{Z_{1}^{'(k)}}\right]}
\rightarrow 0,\quad\text{as $k\rightarrow\infty$}.
\end{equation}
Since, $E\exp(\theta Z_1)<\infty$ (in particular $E[Z_{i}\log^+ Z_{1}]<\infty$),
we have $\lim Z_{i}^{'(k)}/m_{i}^{k}>0$ a.s.
and hence
\begin{equation}
\frac{Z_{2}^{'(k)}}{Z_{1}^{'(k)}}
=\frac{m_{2}^{k}}{m_{1}^{k}}\cdot\frac{Z_{2}^{'(k)}/m_{2}^{k}}{Z_{1}^{'(k)}/m_{1}^{k}}
\rightarrow 0,\nonumber
\end{equation}
as $k\rightarrow\infty$, which implies that
\begin{equation}
\liminf_{k\rightarrow\infty}
E \left[m_2^k\left(\frac{1}{Z_{2}^{'(k)}}-\frac{1}{Z_{1}^{'(k)}}\right)\right]
=\liminf_{k\rightarrow\infty}
E\left[\frac{m_{2}^{k}}{Z_{2}^{'(k)}}\left(1-\frac{Z_{2}^{'(k)}}{Z_{1}^{'(k)}}\right)\right]
\geq E\left[\frac{1}{W_2}\right]>0.\nonumber
\end{equation}
Finally, notice that
\begin{equation}
m_{2}^{k}E\left[\left(\frac{1}{Z_{1}^{(k)}}-\frac{1}{Z_{2}^{(k)}}\right)1_{Z_{1}^{(k)}<Z_{2}^{(k)}}\right]
=E\left[\frac{m_{2}^{k}}{Z_{1}^{(k)}}\left(1-\frac{Z_{1}^{(k)}}{Z_{2}^{(k)}}\right)1_{Z_{1}^{(k)}<Z_{2}^{(k)}}\right]
\leq E\left[\frac{m_{2}^{k}}{Z_{1}^{(k)}}\right].\nonumber
\end{equation}
Therefore, we proved \eqref{gotozero}.
Given any $\beta>23/4$ we can choose $\delta>0$ such that $23/4+\delta<\beta$ and then choose
$k=k(P_{1},P_{2})$ large enough so that the maximum in \eqref{maxoftwo} equals
$23/4+\delta$.
\end{proof}

Finally, let us sketch a proof of Theorem \ref{KThm}.

\begin{proof}[Proof of Theorem \ref{KThm}]
We begin with the independent sequences $\{U_i\}_{i\geq 1}$,
$\{Z_{1,k}\}_{k\geq 1}$, $\{Z_{2,k}\}_{k\geq 1}$ and $\{(\widetilde{Z'}_{1,k},\widetilde{Z'}_{2,k})\}_{k\geq 1}$
where the first three have the same meaning as in Section \ref{CouplingSection}
and  $\{(\widetilde{Z'}_{1,k},\widetilde{Z'}_{2,k})\}_{k\geq 1}$  are i.i.d. copies
of $\left((Z_{1}^{(1)},\hdots,Z_{1}^{(\ell)}),(Z_{2}^{(1)},\hdots,Z_{2}^{(\ell)})\right)$, the latter
having the same meaning as in the
statement of Theorem \ref{KThm}. We shall write
$\widetilde{Z'}_{i,k}=(Z_{i,k}^{(1)},\hdots,Z_{i,k}^{(\ell)})$ for $i=1,\ 2$.

We start both walks at the roots and when $X^{(i)}$ visits the $j$th
distinct site at level $k$ for the first time, we assign
$Z_{i,k+1}^{(j)}$ many children to that site for $i=1,\ 2$ and $j\leq l$. If
one of the walks, say $X^{(1)}$ is visiting the $j$th distinct
site at level $k$ for the first time where $j>\ell$, then we assign
$Z_{1,i}$ many children to that site for some $i$ for which $Z_{1,i}$
has not been used before. At time $n$, we make the transition
using the two rules explained in Section \ref{CouplingSection}
according as the number of children of $X_n^{(1)}$ is larger
or smaller than the number of children of $X_n^{(2)}$.

If $|\mathcal{B}|=k$, we have

(i) $0\leq |X_{\tau_{1}}^{(1)}|-|X_{\tau_{1}}^{(2)}|\leq 2\ell$ when $k\leq\ell$.

(ii) $|X_{\tau_{1}}^{(1)}|-|X_{\tau_{1}}^{(2)}|\geq -2(k-\ell)$ when $k\geq \ell+1$.

This can be argued as follows. Assume that $\mathcal{B}=\{i_1,\hdots,i_k\}$
where $k\geq \ell$. If $|X_j^{(1)}|<|X_j^{(2)}|$ for some $j\leq i_{\ell}$, define
$j_{\ast}:=\min\{i : |X_i^{(1)}|<|X_i^{(2)}|\}$. Then $|X_{j_{\ast}-1}^{(1)}|=|X_{j_{\ast}-1}^{(2)}|$.
Since $j_{\ast}-1<i_{\ell}$, none of the walks has visited any of the levels more than $\ell$ times up till
time $j_{\ast}-1$.
We also have $\min\{Z_{1,k}^{(1)},\hdots,Z_{1,k}^{(\ell)}\}\geq\max\{Z_{2,k}^{(1)},\hdots,Z_{2,k}^{(\ell)}\}$
and hence the number of offsprings of $X_{j_{\ast}-1}^{(1)}$ is not smaller than
the number of offsprings of $X_{j_{\ast}-1}^{(2)}$. But then $|X_{j_{\ast}}^{(1)}|\geq |X_{j_{\ast}}^{(2)}|$, a contradiction.
Hence $|X_j^{(1)}|\geq |X_j^{(2)}|$ whenever $j\leq i_{\ell}$, this implies
the claims in (i) and (ii) stated above. A similar argument
can be given for the case $|\mathcal{B}|<\ell$.

So if we carry out an analysis similar to the one given in the Proof of Theorem \ref{MainThm},
then instead of (\ref{finishing}), we shall get
\begin{align}
&\tilde{E}\left[|X_{\tau_{1}}^{(1)}|-|X_{\tau_{1}}^{(2)}|\right]
\\
&\geq 2\left(\frac{\beta}{\beta+1}\right)^{3}E\left[\frac{(Z_{1}'-Z_{2}')}{4Z_{2}'Z_{1}'\beta}\right]
-\frac{c\cdot 27^{\ell+1}\ell^2}{4^{\ell+1}(\beta-1)(\beta+1)^{\ell}}
E\left[\left(\frac{1}{Z_{1}}-\frac{1}{Z_{2}}\right)1_{Z_{1}<Z_{2}}\right]\nonumber
\\
&\hskip130pt\cdot\sum_{k=\ell+1}^{\infty}
16k(k-\ell)(3k+2)^{2}\left(\frac{27}{4(1+\beta)}\right)^{k-\ell-1},\nonumber
\end{align}
and (\ref{finishingI}) can be modified similarly.
\end{proof}

\begin{proof}[Proof of Corollay \ref{morekids}]

The proof is  an extension and almost the same as the proof of Theorem \ref{MainThm}.
One needs to couple the two random walks on $GW(P_1)$ and $GW(P_2)$,
with a $d\beta$-random walk on $\mathbb{Z}_{\geq 0}$.
Formally we re-define the walk $Y_n$ as $Y_{0}:=0$ and for $n\geq 1$,
\begin{equation}
Y_{n}:=\sum_{i=1}^{n}\left\{1_{U_{i}>\frac{1}{d\beta+1}}-1_{U_{i}\leq\frac{1}{d\beta+1}}\right\},\quad n\in\mathbb{N}.\nonumber
\end{equation}

The walk on  $GW(P_1)$ and $GW(P_2)$ should also be changed accordingly.
Since this is similar to our previous argument, we omit it.
\end{proof}


\section*{Acknowledgements}

The three authors wish to thank G\'erard Ben Arous for suggesting this problem and the useful discussions
that they had with him. The authors also wish to thank Vladas Sidoravicius for helpful discussions.
Finally, the authors thank Alexander Fribergh for reading the first draft of the paper and Elie A\"{i}d\'{e}kon
for private communications.
The three authors are all supported by MacCracken Fellowship at New York University.
In addition, Sanchayan Sen is supported by NSF grant DMS-1007524, and Lingjiong Zhu is supported by NSF grant DMS-0904701 and DARPA grant.

\bibliographystyle{elsarticle-harv}
\bibliography{mybib}

\end{document}